\documentclass[12pt]{amsart}
\usepackage{amsmath,amsfonts,amsthm,amssymb}
\usepackage[all]{xypic}
\newcommand{\scrF}{\mathcal{F}}
\newcommand{\scrN}{\mathcal{N}}
\newcommand{\scrU}{\mathcal{U}}

\newcommand{\fb}{\mathfrak{b}}

\newcommand{\bs}{\mathbf{s}}
\newcommand{\bu}{\mathbf{u}}

\newcommand{\bbF}{\mathbb{F}}
\newcommand{\bbQ}{\mathbb{Q}}
\newcommand{\bbR}{\mathbb{R}}
\newcommand{\bbZ}{\mathbb{Z}}

\newcommand{\hzp}{H^*_{\bbZ/p}(\mathrm{pt})}

\newtheorem{dummy}{Dummy}[section]

\newtheorem{proposition}[dummy]{Proposition}
\newtheorem{lemma}[dummy]{Lemma}
\newtheorem{corollary}[dummy]{Corollary}
\newtheorem{theorem}[dummy]{Theorem}
\theoremstyle{definition}

\newtheorem{remark}[dummy]{Remark}

\title{A topological approach to induction theorems in Springer theory}
\author{David Treumann}
\date{September 16, 2008}

\SelectTips{cm}{}
\newdir^{ (}{{}*!/-5pt/@^{(}}
\newcommand{\har}{\ar@{^{ (}->}}

\newcommand{\git}{\backslash \! \backslash}
\newcommand{\bul}{$\bullet$}
\newcommand{\tscrN}{\tilde{\scrN}}
\newcommand{\reg}{\mathrm{rss}}

\begin{document}
\maketitle

\section{Introduction}

The Springer correspondence is a connection between the geometry of the space of unipotent elements $\scrN$ in an algebraic group $G$ and the representation theory of the Weyl group $W$ associated to $G$.  Representations of $W$ are constructed on the cohomology of the fibers of Springer's resolution of singularities $\tscrN \to \scrN$.  It has been clear for a long time that representations can be constructed this way over any ground ring, but the focus of the theory has overwhelmingly been on representations in characteristic zero.  Recently Juteau \cite{juteau} has begun a systematic study of the modular version of this theory, extending the intersection homology techniques used by Lusztig, Borho, and MacPherson \cite{lusztig2, borhomacp} and studying the extent to which the Brauer-Cartan (or ``cde'') triangle can be seen in the geometry of $\scrN$.

This paper is another contribution to the modular gap in Springer theory.  Our perspective is topological: where Juteau approaches modular representations via the techniques of Lusztig-Borho-MacPherson, we follow the approach of Grothendieck-Brieskorn-Slodowy.  We use this construction to produce ``induction theorems'' which relate the Springer correspondences for $G$ and for a subgroup of $G$.

\subsection{Modular representations via Rossmann's homotopy action}
A major part of this paper is a review of and advocacy for an approach to Springer theory due to Rossman \cite{rossmann}.  Write $F_\nu \subset \tscrN$ for the fiber of Springer's resolution over $\nu \in \scrN$.  The existence of the $W$-action on $H^*(F_\nu)$, constructed first by Springer in 1976, came as a surprise especially because in simple examples it is easy to see that it cannot possibly be induced by an action of $W$ on the space $F_\nu$.  For a long time it was unknown whether the action on $H^*(F_\nu)$ could be lifted to an action on the homotopy type of $F_\nu$ (though early progress was made in \cite{KL}).  Rossmann gave an affirmative solution to this problem in 1993.

In the first five sections of this paper we give a self-contained account of Rossmann's construction.  We produce an action of the braid group $B_W$ associated to $W$ on a $\dim(G)$-dimensional manifold-with-partial-boundary $\scrU \times C_\nu$ which is homotopy equivalent to $F_\nu$, and show that the kernel of the projection $B_W \to W$ acts by transformations homotopic to the identity.  This is a little weaker than producing a space homotopy equivalent to $F_\nu$ on which $W$ acts---this is still an open problem as far as I know.  Nevertheless the construction does give an action of $W$ on every functorial homotopy invariant of $F_\nu$, in particular on mod $p$ cohomology.

The spaces $\scrU$ and $C_\nu$ are easy to describe: $\scrU$ is the contractible universal cover of the set of regular elements in a small ball around $1 \in W\git T$, and $C_\nu$ is a regular neighborhood with boundary of $F_\nu$ in $\tscrN$.  Rossmann's construction follows a standard construction (due originally to Grothendieck, Brieskorn, and Slodowy \cite{brieskorn, slodowy}) of the Springer representations via nearby cycles.  In that approach, the essential point is that a Springer fiber has the same cohomology as the appropriate Milnor fibers of the projection $f:G \to W \git T$.  (This is expressed sheaf-theoretically as an isomorphism between the nearby cycles sheaf of $f$ and the direct image of the constant sheaf under the Springer map $\tscrN \to \scrN$.  See \cite{grinberg} for a nice explanation.)  The essential point of Rossmann's refined construction is that these Milnor fibers are actually homeomorphic to $C_\nu$.

\subsection{Induction theorems}
Let $H \subset G$ be a connected reductive subgroup of $G$.  Suppose that $\nu$ lies in $H$.  Then we may study the Springer representations for the pair $\nu \in H$ and for the pair $\nu \in G$.  A basic question is how the representations of $W_H$ on $H^*(F_{\nu,H})$ and of $W_G$ on $H^*(F_{\nu,G})$ are related.  In case $H$ is a Levi subgroup, Alvis and Lusztig \cite{alvislusztig,lusztig} showed that the characters of the $W_G$-modules $\mathrm{Ind}_{W_H}^{W_G} H^*(F_{\nu,H})$ and $H^*(F_{\nu,G})$ are the same.

\begin{theorem}[Alvis-Lusztig, \cite{alvislusztig,lusztig}]
\label{thm:alvislusztig}
Let $G$ be a reductive group, let $L$ be a Levi subgroup, and let $\nu$ be a unipotent element of $L$.  Let $W_L \subset W_G$ denote the Weyl groups of $L$ and $G$, and let $F_{\nu,L}$ and $F_{\nu,G}$ denote the Springer fibers associated to  $\nu$ in $L$ and $G$.  There is a filtration by $W_G$-submodules of
$$\text{\rm Ind}_{W_L}^{W_G} (\bigoplus_{i \in \bbZ} H^i(F_{\nu,L};\bbQ))$$
so that the associated graded satisfies 
$$\mathrm{gr}^k(\text{\rm Ind}_{W_L}^{W_G} (\bigoplus_{i \in \bbZ} H^i(F_{\nu,L};\bbQ))) \cong H^k(F_{\nu,G};\bbQ)$$
\end{theorem}

In this paper we use Rossmann's construction to give a topological explanation for this.  Our proof is reminiscent of topological work on the ``fundamental lemma'' in the Langlands program \cite{fundlem1,fundlem2}: we use localization in circle-equivariant cohomology to make comparisons.  This technique, using $\bbZ/p$-equivariant localization instead, applies in the modular setting as well.  We will show

\begin{theorem}
\label{thm:ALmodular}
Let $G$ be a reductive group, let $\nu \in G$ be a unipotent element, and let $g \in G$ be an element of order $p$ ($p$ a prime) which is contained in the connected component of the centralizer of $\nu$.  Let $Z \subset G$ be the centralizer of $g$ in $G$; note that $\nu$ is contained in $Z$.  Suppose furthermore that $Z$ is connected (this is true, for instance, if $G$ is simply connected.)  Let $W_Z$ and $W_G$ denote the Weyl groups of $Z$ and $G$, and let $F_{\nu,Z}$ and $F_{\nu,G}$ be the Springer fibers associated to $\nu$ in $Z$ and $G$.  There is a filtration by $W_G$-modules of
$$\text{\rm Ind}_{W_Z}^{W_G} (\bigoplus_{i \in \bbZ} H^i(F_{\nu,Z};\bbF_p))$$
whose associated graded satisfies
$$\mathrm{gr}^k(\text{\rm Ind}_{W_Z}^{W_G} (\bigoplus_{i \in \bbZ} H^i(F_{\nu,Z};\bbF_p))) \cong  H^k(F_{\nu,G};\bbF_p)$$
\end{theorem}

\section{Springer glossary}
\label{sec:glossary}

For our purposes, Springer theory is the geometry of the spaces and maps in this diagram:
\begin{equation}
\label{eq-grinberg}
\SelectTips{cm}{}
\xymatrix{
\tscrN \har[r]^{\tilde{j}} \ar[d]^p & \tilde{G} \ar[r]^{\tilde{f}} \ar[d]^{q} & T \ar[d]^{\pi} \\
\scrN \har[r]^j & G \ar[r]^f &W \backslash \! \backslash T
}
\end{equation}
Let us briefly indicate what these spaces are; for more details we refer to \cite{grinberg}.
\begin{enumerate}
\item[\bul] $G$ is a fixed complex reductive Lie group.
\item[\bul] $T$ is the ``universal maximal torus.''  That is, $T$ is an algebraic torus which is naturally identified with $B/[B,B]$ for every Borel subgroup $B \subset G$.
\item[\bul] $W$ is the Weyl group of $G$, and $W \git T$ is the space of $W$-orbits in $T$.
\item[\bul] $\tilde{G}$ is the Grothendieck alteration of $G$: the space of pairs $(x,B)$ where $B \subset G$ is a Borel subgroup and $x$ is an element of $B$.
\item[\bul] $\scrN$ is the set of unipotent elements of $G$
\item[\bul] $\tscrN$ is the Springer resolution of $\scrN$: the space of pairs $(x,B)$ where $B$ is a Borel subgroup and $x$ is an element of $[B,B]$.
\end{enumerate}
The maps $j$ and $\tilde{j}$ are just the inclusions.  The maps $p$ and $q$ are given by $(x,B) \mapsto x$ and $\pi$ is the quotient map.  The map $\tilde{f}$ carries $(x,B)$ to its image in $B /[B,B] \cong T$, and $f$ is the unique map making the right-hand square commute.  Alternatively, letting $G$ act on itself by conjugation, one may define $f$ to be the GIT quotient map $G \to G \git G$ composed with a canonical isomorphism $W \git T \cong G \git G$.

Each of the spaces $\tilde{G}, G,T$ and $W \git T$ have dense open subsets of \emph{regular semisimple} elements; we write them as $\tilde{G}^{\reg} \subset \tilde{G}$, $G^\reg \subset G$, etc.  Explicitly, we have
\begin{enumerate}
\item[\bul] $G^{\reg}$ is the set of regular semisimple elements in $G$.
\item[\bul] $\tilde{G}^\reg$ is the set of pairs $(x,B) \in \tilde{G}$ where $x$ is a regular semisimple element.
\item[\bul] $T^{\reg}$ is the set of regular semisimple elements in $T$.
\item[\bul] $(W \git T)^{\reg}$ is the set of $W$-orbits of regular semisimple elements in $T$.
\end{enumerate}

Each of the maps $\tilde{f}, f,q,\pi$ restricts to a submersion on the set of regular elements, and the square
$$\xymatrix{
\tilde{G}^\reg \ar[r] \ar[d] & T^\reg \ar[d] \\
G^{\reg} \ar[r] & (W \git T)^\reg
}$$
is cartesian.  In particular, the vertical arrows are covering spaces with deck group $W$.

\begin{remark}
If $X$ is a subset of one of $\tilde{G}, G,T$ or $W \git T$ we will write $X^\reg$ for the intersection of $X$ with $\tilde{G}^\reg, G^\reg,T^\reg$ or $(W \git T)^\reg$ respectively.
\end{remark}

\section{Metric structures and a trivialization of $\tilde{f}$}
\label{sec-trivial}

Fix a maximal compact subgroup $K$ of $G$, and endow $G$ with a Hermitian metric which is left and right invariant under $K$.  This structure provides, for each Borel subgroup, a canonical diffeomorphism $B \cong T \times [B,B]$.  Indeed, $B \cap K$ is a maximal compact torus of $B$, and its Zariski closure in $B$ gives a section of the principal $[B,B]$-bundle $B \to T$.  Alternatively, the diffeomorphism $B \cong T \times [B,B]$ can be obtained by exponentiating the orthogonal decomposition of the Lie algebra
$\fb = [\fb,\fb] \oplus [\fb,\fb]^\perp$.  Let us write $[B,B]^\perp$ for the connected subgroup of $B$ whose Lie algebra is $[\fb,\fb]^\perp$, or equivalently for the Zariski closure of $B \cap K$.

Denote the projection $B \to [B,B]$ by $\bu_B$ and the projection $B \to [B,B]^\perp$ by $\bs_B$ (the u and s stand for ``unipotent'' and ``semisimple'').  The $\bu_B$ assemble to a trivialization of $\tilde{f}$.  That is, the map
$$\tilde{G} \to \tscrN \times T:(x,B) \mapsto ((\bu_B(x),B),\tilde{f}(x,B))$$
is a diffeomorphism and commutes with the projections to $T$.

\begin{remark}
We will actually only make use of the fact that $\tilde{f}:\tilde{G} \to T$ is a submersion -- this is a weaker statement because $\tilde{f}$ is not proper.
\end{remark}

The Hermitian structure on $G$ induces one on the universal torus $T$, via the isomorphism $T \cong [B,B]^\perp$; the metric is independent of $B$.  Let us write $d(x,y)$ for the induced distance function on these spaces.  Let us also define a metric on the space $W \git T$, by setting
$$d(W\cdot x, W\cdot y) = \inf_{w \in W} d(w \cdot x,y)$$
for the distance between the $W$-orbit of $x$ and the $W$-orbit of $y$.

For each $\delta > 0$, write $B(\delta,T)$ for the open $\delta$-ball centered at the identity in $T$, and write $B(\delta,W\git T)$ for the open $\delta$-ball centered at the identity orbit in $W \git T$.
$$\begin{array}{c}
B(\delta,T) := \{x \in T \mid d(x,1) < \delta\} \\
B(\delta,W\git T) := \{W \cdot x \in W \git T \mid d(W \cdot x,1) < \delta\}
\end{array}
$$
Note that we have $d_T(x,1) = d_{W \git T}(W \cdot x,1)$ for all $x \in T$.  Thus the image of $B(\delta,T)$ under $\pi$ is $B(\delta,W\git T)$.

\section{Some spaces near a unipotent element}
\label{sec:spacesnear}

Fix a unipotent element $\nu \in \scrN$, and let $\epsilon$ and $\delta$ be positive real numbers with $\delta \ll \epsilon \ll 1$.  Define the following spaces:

\begin{enumerate}
\item[\bul] Let $F_{\nu} \subset \tscrN$ be the Springer fiber over $\nu$:
$$F_\nu := p^{-1}(\nu)$$
\item[\bul] Let $E_{\nu}(\epsilon,\delta) \subset G$ be the intersection of a closed $\epsilon$-ball around $\nu$ with the inverse image of an open $\delta$-ball around $1 \in W \git T$.  If $B$ is any Borel subgroup containing $x$ then we may write this as
$$E_\nu := \{x \in G \mid d(x,\nu) \leq \epsilon \text{ and } d(\bs_B(x), 1) < \delta\}$$
\item[\bul] Let $D_{\nu}(\epsilon,\delta) \subset \tilde{G}$ be the inverse image of $E_{\nu}$ under $q$:
$$D_{\nu} := \{(x,B) \in \tilde{G} \mid d(x,\nu) \leq \epsilon \text{ and } d(\bs_B(x),1)<\delta\}$$
\item[\bul] Let $C_\nu(\epsilon) \subset \tscrN$ be the intersection of $D_\nu$ with $\tscrN$:
$$C_{\nu} := \{(x,B) \in \tscrN \mid d(x,\nu) \leq \epsilon\}$$
(Since $\tscrN \to \scrN$ is a resolution of singularities, a standard argument shows that $C_\nu$ deformation retracts onto $F_\nu$.  See proposition \ref{CnuFnu}.)
\end{enumerate}

$E_\nu$ and $D_\nu$ fit into the following commutative square:
$$\xymatrix{
D_\nu \ar[r] \ar[d] & B(\delta,T) \ar[d] \\
E_\nu \ar[r] & B(\delta, W \git T)
}$$

This square becomes cartesian when the maps are restricted to the regular parts of each space:
$$\xymatrix{
D_\nu^\reg \ar[r] \ar[d] & B(\delta,T)^\reg \ar[d] \\
E_\nu^\reg \ar[r] & B(\delta,W \git T)^\reg
}$$

\subsection{The action of the compact centralizer}

Write $K(\nu) = K \cap Z_G(\nu)$ for the maximal compact subgroup of the centralizer of $\nu$ which is contained in $K$.  The distance from $\nu$ with respect to the hermitian metric on $G$ is invariant under $K(\nu)$, and therefore $K(\nu)$ acts on the spaces $F_\nu,E_\nu(\epsilon,\delta),D_\nu(\epsilon,\delta)$, and $C_\nu(\epsilon)$.  Moreover, the fibers of $D_\nu \to B(\delta,T)$ and $E_\nu \to B(\delta,W \git T)$ are preserved by this action.  We also have

\begin{proposition}
\label{CnuFnu}
For $\epsilon$ sufficiently small, the inclusion $F_\nu \to C_\nu(\epsilon)$ admits a $K(\nu)$-equivariant deformation retraction. 
\end{proposition}  

\begin{proof}
The function 
$$\rho:\tscrN \to \bbR:(x,B) \mapsto d(x,\nu)^2$$
that measures distance from $\nu$ is $K(\nu)$-invariant and proper, and has $0$ as an isolated critical value.  Thus, its gradient vector field integrates to a $K(\nu)$-equivariant deformation retraction from $C_\nu(\epsilon) = \rho^{-1}[0,\epsilon^2]$ to $F_\nu = \rho^{-1}(0)$, for $\epsilon$ sufficiently small.  
\end{proof}

\section{The key lemma and its consequences}

This section is devoted to the proof of the following lemma:

\begin{lemma}
\label{keylemma}
For some $\epsilon_0>0$ and $\delta_0 = \delta_0(\epsilon)>0$, the map $D_\nu(\epsilon,\delta) \to B(\delta,T)$ is a trivial fiber bundle whenever $\epsilon < \epsilon_0$ and $\delta < \delta_0(\epsilon)$.  Moreover the trivialization $D_\nu(\epsilon,\delta) \cong C_\nu(\epsilon) \times B(\delta,T)$ can be chosen $K(\nu)$-equivariant.
\end{lemma}

We will use the following very special cases of Thom's isotopy lemmas:

\begin{theorem}[Thom's first isotopy lemma]
Let $M$ be a manifold with boundary, let $B \subset \bbR^n$ be an open ball, and let $u:M \to B$ be a smooth proper surjective mapping.  Suppose that $u$ is submersive on both the interior and the boundary of $M$.  Then $u$ is a topologically trivial fiber bundle.
\end{theorem}

\begin{theorem}[Thom's second isotopy lemma]
Let $M$ and $N$ be manifolds with boundary, and let $v:N \to M$ be a proper map which exhibits $N$ as a locally trivial fiber bundle over $M$.  Let $B \subset \bbR^n$ be an open ball, and let $u:M \to B$ be a smooth proper surjective mapping.  Suppose that $u$ is submersive on both the interior and the boundary of $M$, so that per the first isotopy lemma $u$ admits a trivialization over $B$.  Then the map $v:N \to M$ admits a trivialization over $B$: for a given point $b \in B$, and letting $w$ denote the natural map $(u \circ v)^{-1}(b) \to u^{-1}(b)$, there are homeomorphisms making the following diagram commute
$$\xymatrix
{
N \ar[r]^{\cong \quad \quad \quad} \ar[d]_v & B \times (u \circ v)^{-1}(b) \ar[d]^{1_B \times w} \\
M \ar[r]_{\cong \quad} & B \times u^{-1}(b)
}
$$
Furthermore, all maps in this square commute with the projections to $B$.
\end{theorem}

Proofs of the full isotopy lemmas may be found in \cite{isotopy}, however we remark that the special cases we consider here may be proved by much more elementary means.

\begin{proof}[Proof of lemma \ref{keylemma}]
Let $D_\nu^\circ(\epsilon,\delta) \subset D_\nu(\epsilon,\delta)$ be the open set of those $(x,B) \in D_\nu(\epsilon,\delta)$ with $d(x,\nu) < \epsilon$, and let
$\partial D_\nu(\epsilon,\delta) \subset D_\nu(\epsilon,\delta)$ denote the closed complement of those $(x,B)$ with $d(x,\nu) = \epsilon$.  Define $C_\nu^\circ(\epsilon) \subset C_\nu(\epsilon)$ and $\partial C_\nu(\epsilon)$ similarly.  Let $\rho$ denote the function $\tilde{G} \to \bbR$ defined by
$$\rho(x,B) = d(x,\nu)^2$$
Since $\rho$ and $\rho\vert_{\tscrN}$ are real algebraic and smooth, their critical values are isolated by Sard's theorem.  It follows that there exists a positive real number $\epsilon_0$ such that neither $\rho$ nor $\rho\vert_{\tscrN}$ have critical values in the open interval $(0,\epsilon_0)$.  In particular, $\partial D_\nu(\epsilon,\delta)$ (and $\partial C_\nu(\epsilon)$) is a manifold for every $\epsilon$ in $(0,\epsilon_0)$.

By the first isotopy lemma, the fact that $D_\nu(\epsilon,\delta) \to B(\delta,T)$ is a trivial fiber bundle is a consequence the following claim: there exists $\delta_0$ such that the projections $D_\nu^\circ(\epsilon,\delta) \to B_{\delta}(T)$ and $\partial D_\nu(\epsilon,\delta) \to B_{\delta}(T)$ are submersions for all $\epsilon < \epsilon_0$ and $\delta < \delta_0$.  Furthermore, applying Thom's second isotopy lemma to the case where $M = D_\nu(\epsilon,\delta)$,  $N = K(\nu) \times D_\nu(\epsilon,\delta)$, and $v:N \to M$ is the action map gives a $K(\nu)$-equivariant trivialization.  Thus, let us prove the claim.

Since $q:\tilde{G} \to T$ is a submersion (by section \ref{sec-trivial}) and $D^\circ_\nu(\epsilon,\delta)$ is open in $\tilde{G}$, the map $D_\nu^\circ(\epsilon,\delta) \to B(\delta,T)$ is also a submersion for any positive $\epsilon$ and $\delta$.  
We are left with showing that $q\vert_{\partial D_\nu(\epsilon,\delta)}$ is a submersion for $\epsilon < \epsilon_0$ and $\delta$ sufficiently small.  
To show that $q\vert_{\partial D_\nu(\epsilon,\delta)}$ does not have critical values in a sufficiently small neighborhood of $1 \in T$, it suffices to show (since $q\vert_{\partial D_\nu(\epsilon,\delta)}$ is proper and the set of critical values is closed) that $0$ itself is not critical -- i.e. that at any point $(x,B) \in \tscrN = q^{-1}(0)$, the vertical tangent space to $q$ at $(x,B)$ and the tangent space to $\partial D_\nu(\epsilon,\delta)$ at $(x,B)$ are transverse.  

Let us call these tangent spaces $T_1$ and $T_2$, respectively.  
Since $\partial D_\nu(\epsilon,\delta)$ is of real codimension one in $\tilde{G}$, it is equivalent to showing that $T_2$ does not contain $T_1$.  But  $T_1$ is the tangent space to $\tscrN$ and $T_2$ is the kernel of $d\rho_{x,B}$, so if $T_1 \subset T_2$ then $(x,B)$ is a critical point of $\rho\vert_{\tscrN}$, which contradicts the assumption that $\epsilon < \epsilon_0$.  This completes the proof.

\end{proof}

We immediately have the following corollaries:

\begin{corollary}
\label{cor:milnor}
If $\epsilon_0$ and $\delta_0$ are as in lemma \ref{keylemma}, then for $\epsilon < \epsilon_0$ and $\delta < \delta_0(\epsilon)$, the map $E_\nu(\epsilon,\delta)^\reg \to B(\delta,W \git T)^\reg$ is a $K(\nu)$-equivariant locally trivial fiber bundle.
\end{corollary}

In particular, the Borel construction of $E_\nu(\epsilon,\delta)$ with respect to any subgroup $K' \subset K(\nu)$ may be exhibited as a fiber bundle over $B(\delta,W \git T)^\reg$, whose fibers are Borel constructions of the $K'$-space $C_\nu(\epsilon)$.  Indeed, letting $EK'$ denote a contractible space on which $K'$ acts freely, the map $K' \backslash (E_\nu(\epsilon,\delta) \times EK') \to B(\delta,W\git T)^\reg$ that takes the $K'$-orbit of $(x,y)$ to $f(x)$ is well-defined and provides such a fibration.  We will use this fact in the next section.

\begin{corollary}
\label{cor:slodowy}
Let $A$ be a commutative ring.  With $\epsilon$ and $\delta$ as above, let $\phi$ denote the map $E_\nu^\reg \to B(\delta,W \git T)^\reg$.  The higher direct image sheaves $R^i \phi_* A$ are locally constant on $B(\delta,W\git T)^\reg$.  Moreover, the stalks are isomorphic to the cohomology sheaves $H^i(F_\nu;A)$ with coefficients in $A$, and the monodromy action $\pi_1(B(\delta,W \git T)^\reg) \to \mathrm{Aut}(H^i(F_\nu;A))$ factors through a quotient $\pi_1(B(\delta,W \git T)^\reg) \to W$. 
\end{corollary}

It is shown in \cite{hotta} that when $A = \bbQ_\ell$, this action of $W$ on $H^i(F_\nu)$ coincides with the original action of Springer after tensoring with the sign representation of $W$.

\subsection{Weyl group action on Springer fibers at the level of homotopy}
Let us indicate how this result lifts Springer's action of $W$ on the cohomology of $F_\nu$ to an action of $W$ on the homotopy type of $F_\nu$.  Let $\scrU$ be the universal cover of $B_\delta(W \git T)^\reg$ (equivalently: of $B_\delta(T)^\reg$); it is well-known that $\scrU$ is contractible and that the deck group of $\scrU$ is $B_W$, the braid group associated to $W$ (this was finally proved for all reflection groups in \cite{deligne}).  It follows that $B_W$ acts on the pullback $E_\nu^\reg \times_{B_\delta(W\git T)^\reg}\scrU$ of $E_\nu^\reg$ to $\scrU$.
In fact, it acts freely.  The lemma provides a homeomorphism between this pullback and $C_\nu \times \scrU$, and by proposition \ref{CnuFnu}, this space is homotopy equivalent to the Springer fiber $F_\nu$.  To complete the construction, we need to check that the subgroup $B_W'$ of ``pure braids'' (i.e. the kernel of the projection $B_W \to W$) acts on $\scrU \times C_\nu$ by transformations homotopic to the identity.  This also follows from the lemma: $C_\nu \times \scrU/{B_W'}$ is trivial over $\scrU/{B_W'} \cong B_\delta(T)^\reg$.

\section{Induction theorems}
In this section we will prove theorems \ref{thm:alvislusztig} and \ref{thm:ALmodular}.  To do this we will need to consider the spaces $F,E,D,C$ defined in section \ref{sec:spacesnear} for different reductive groups $G$.  To keep them straight we need to complicate the notation introduced in sections \ref{sec:glossary} and \ref{sec:spacesnear} with subscripts.

\begin{enumerate}
\item[\bul] Write $T_G, W_G, \scrN_G$, and $\tscrN_G$ for the universal torus, unipotent cone, and Springer resolution associated to the complex connected reductive algebraic group $G$.
\item[\bul] Write $F_{\nu,G}, E_{\nu,G}(\epsilon,\delta), D_{\nu,G}(\epsilon, \delta)$, and $C_{\nu,G}(\epsilon)$ for the spaces associated to $G$ defined in section \ref{sec:spacesnear}.
\item[\bul] If $G$ is endowed with a hermitian metric, write $K_G$ for the maximal compact which preserves it.  Write $K_G(\nu)$ for the intersection of $K_G$ with the centralizer of the unipotent element $\nu$.
\end{enumerate}

\begin{proof}[Proof of theorem \ref{thm:alvislusztig}]
Since any maximal torus of $L$ is also a maximal torus of $G$, we may identify the universal tori $T_L$ and $T_G$, so we will omit the subscript.  For appropriate $\epsilon$ and $\delta$, $\delta \ll \epsilon \ll 1$, we have by corollary \ref{cor:milnor} a commutative square
$$
\xymatrix{
E_{\nu,L}(\epsilon,\delta)^\reg \ar[r] \ar[d] & E_{\nu,G}(\epsilon,\delta)^\reg \ar[d] \\
B(\delta,W_L\git T)^\reg \ar[r] & B(\delta,W_G\git T)^\reg
}
$$
(Let us fix $\epsilon$ and $\delta$, and for the rest of the proof omit them from the notation for $E_{\nu,L}, E_{\nu,G}$.)
Let $U(1) \subset G$ be a compact circle whose centralizer in $G$ is $L$.  Then $U(1) \subset K_G(\nu)$ and the $U(1)$-fixed points $E_{\nu,G}^{U(1)}$ coincide with $E_{\nu,L}$.  We will now deduce the theorem by standard ``localization'' techniques in $U(1)$-equivariant cohomology, and the degeneration of the Borel spectral sequence.

We may replace the top row of the diagram above by the Borel constructions with respect to $U(1)$.  That is, letting $X_{hU(1)}$ denote the Borel construction of a $U(1)$-space $X$, we have
$$\xymatrix{
(E_{\nu,L}^\reg)_{hU(1)} \ar[r] \ar[d] & (E_{\nu,G}^\reg)_{hU(1)} \ar[d] \\
B(\delta,W_L \git T)^\reg \ar[r] & B(\delta,W_G \git T)^\reg
}$$
Let us denote the vertical maps by $\chi_L$ and $\chi_G$, and the bottom horizontal map by $\iota$.  Define a locally constant sheaf $\scrF_L$ on $B(\delta, W_L \git T)^\reg$ and $\scrF_G$ on $B(\delta, W_G \git T)^\reg$ by
$$\begin{array}{c}
\scrF_L = \bigoplus_{i \in \bbZ} R^i \chi_{L*} \bbQ \\
\scrF_G = \bigoplus_{i \in \bbZ} R^i \chi_{G*} \bbQ
\end{array}
$$
The fibers of $\scrF_L$ (resp. $\scrF_G$) are isomorphic to the equivariant cohomology groups $H^*_{U(1)}(F_{\nu,L};\bbQ)$ (resp. $H^*_{U(1)}(F_{\nu,G};\bbQ)$).
Fix an isomorphism of graded rings $H^*_{U(1)}(\mathit{pt};\bbQ) \cong \bbQ[t]$, where $t$ has degree $2$; then $\scrF_L$ and $\scrF_G$ are sheaves of $\bbQ[t]$-modules.  Since $E_{\nu,L}^\reg$ is the $U(1)$-fixed set in $E_{\nu,G}^\reg$, the localization theorem for $U(1)$-equivariant cohomology (see e.g. \cite[Theorem 4.4]{quillen}) provides an isomorphism of graded $\bbQ[t,t^{-1}]$-modules
$$\scrF_G \otimes_{\bbQ[t]} \bbQ[t,t^{-1}] \cong \iota_* \scrF_L \otimes_{\bbQ[t]} \bbQ[t,t^{-1}]$$
Translating this into the language of $W_G$-modules, we have
$$H^*_{U(1)}(F_{\nu,G},\bbQ) \otimes_{\bbQ[t]} \bbQ[t,t^{-1}] \cong \text{Ind}_{W_L}^{W_G} H^*_{U(1)}(F_{\nu,L};\bbQ) \otimes_{\bbQ[t]} \bbQ[t,t^{-1}]$$
Since $U(1)$ acts trivially on $F_{\nu,L}$, the group on the right is isomorphic to $\text{Ind}_{W_L}^{W_G} H^*(F_{\nu,L};\bbQ) \otimes_\bbQ \bbQ[t,t^{-1}]$.

Consider the Leray spectral sequence for the composition of maps
$$(E_{\nu,G}^\reg)_{hU(1)} \to BU(1) \times B(\delta,W_G\git T)^\reg \to B(\delta, W_G \git T)^\reg$$
This is a spectral sequence in the category of locally constant sheaves on $B(\delta, W_G \git T)^\reg$, whose stalk at a point is the Borel spectral sequence
$$E_2^{ij} = H^i_{U(1)}(\mathit{pt}) \otimes H^j(F_{\nu,G};\bbQ) \implies H^{i+j}_{U(1)}(F_{\nu,G};\bbQ)$$
In particular, we may regard the Borel spectral sequence as a spectral sequence of $W_G$-modules.  Furthermore, since $H^*(F_{\nu,G};\bbQ)$ is known to vanish in odd degrees (in fact this is true for any coefficient ring, \cite{dpl}), this spectral sequence degenerates.  It follows that the localized spectral sequence
$$E_2^{ij}[t^{-1}]:= \bbQ[t,t^{-1}] \otimes H^j(F_{\nu,G};\bbQ) \implies H^*_{U(1)}(F_{\nu,G};\bbQ) \otimes_{\bbQ[t]} \bbQ[t,t^{-1}]$$
also degenerates.  In other words, $H^*_{U(1)}(F_{\nu,G};\bbQ) \otimes_{\bbQ[t]} \bbQ[t,t^{-1}] \cong \text{Ind}_{W_L}^{W_G} H^*(F_{\nu,L};\bbQ) \otimes_\bbQ \bbQ[t,t^{-1}]$ admits a filtration whose associated graded is $E_2^{ij}[t^{-1}]$.  The conclusion of the theorem follows from restricting to the $i+j = 0$ line of $E_2^{ij}[t^{-1}]$.
\end{proof}

\begin{proof}[Proof of theorem \ref{thm:ALmodular}]
Let $\bbZ/p \subset G$ be the subgroup generated by $g$.  As in the proof of theorem \ref{thm:alvislusztig}, we have $\bbZ/p \subset K_G(\nu)$, and $E_\nu^\reg$ is the set of $\bbZ/p$-fixed points of $E_{\nu,G}^\reg$.  Thus as before, this time by applying the localization theorem for $\bbZ/p$-equivariant cohomology (theorem 4.2 of \cite{quillen}), we have an isomorphism of $W_G$-modules
$$H^*_{\bbZ/p}(F_{\nu,G},\bbF_p) \otimes_{\hzp} \hzp[t^{-1}] \cong \text{Ind}_{W_Z}^{W_G} H^*_{\bbZ/p}(F_{\nu,Z};\bbF_p) \otimes_{\hzp} \hzp[t^{-1}]$$
Here $t$ denotes a polynomial generator in degree $2$ if $p$ is odd, and degree $1$ if $p = 2$.  Once again since $\bbZ/p$ acts trivially on $F_{\nu,Z}$, we have $\text{Ind}_{W_Z}^{W_G} H^*_{\bbZ/p}(F_{\nu,Z};\bbF_p) \otimes_{\hzp} \hzp[t^{-1}] \cong \text{Ind}_{W_Z}^{W_G} H^*(F_{\nu,Z};\bbF_p) \otimes_{\bbF_p} \hzp[t^{-1}]$, and we may complete the proof by showing that the Borel spectral sequence
$$E_2^{ij} = \hzp \otimes H^j(F_{\nu,G};\bbF_p) \implies H^{i+j}_{\bbZ/p}(F_{\nu,G};\bbF_p)$$
degenerates.

We may do this as follows.  Since $g$ is in the connected component of the identity of the centralizer of $\nu$, we may find a circle subgroup $U(1) \subset K_G(\nu)$ with $g \in U(1)$.  Then we have a cartesian square
$$
\xymatrix{
(F_{\nu,G})_{h\bbZ/p} \ar[r]^{a} \ar[d]_b & (F_{\nu,G})_{hU(1)}\ar[d]^c \\
B(\bbZ/p) \ar[r]_{d} & BU(1)
}
$$
The morphisms in this square are fiber bundles with fiber $F_{\nu,G}$.  Since both the cohomology groups $H^*(F_{\nu,G},\bbF_p)$ and $H^*(BU(1);\bbF_p)$ vanish in odd degrees, the derived pushforward sheaf $Rc_* \bbF_p$ is formal, i.e. isomorphic to a direct sum of its cohomology sheaves.  By the proper base change theorem, the same is true for $Rb_* \bbF_p$.  This implies that the Leray spectral sequence of the map $b$ and the sheaf $\bbF_p$ degenerates.  But this Leray spectral sequence is exactly the Borel spectral sequence we are considering.  This completes the proof.
\end{proof}

\begin{remark}
In case $g$ is not contained in the identity component of the centralizer of $\nu$, we may still conclude from the existence of the Borel spectral sequence that 
$$\mathrm{gr}^k(\text{\rm Ind}_{W_Z}^{W_G} (\bigoplus_{i \in \bbZ} H^i(F_{\nu,Z};\bbF_p)))$$ is a subquotient of $H^k(F_{\nu,G};\bbF_p)$, even if this spectral sequence does not degenerate.
\end{remark}

\begin{remark}
Let us give an example (without proof) where the hypotheses of theorem \ref{thm:ALmodular} apply, and one where they do not.  Let $G = Sp(4)$, $Z = SL(2) \times SL(2)$ (note that this is not a Levi subgroup) and $p = 2$.  Then $Z$ is the centralizer of
$$
g = \left(
\begin{array}{cccc}
1 & 0 & 0 & 0 \\
0 & 1 & 0 & 0 \\
0 & 0 & -1 & 0 \\
0 & 0 & 0 & -1
\end{array}
\right)
$$
We may take for $\nu$ the matrix
$$
\left(
\begin{array}{cccc}
1 & 1 & 0 & 0 \\
0 & 1 & 0 & 0 \\
0 & 0 & 1 & 0 \\
0 & 0 & 0 & 1
\end{array}
\right)
$$ 
but not the subregular unipotent
$$
\left(
\begin{array}{cccc}
1 & 1 & 0 & 0 \\
0 & 1 & 0 & 0 \\
0 & 0 & 1 & 1 \\
0 & 0 & 0 & 1
\end{array}
\right)
$$ 

\end{remark}

\emph{Acknowledgements:}  I am very grateful to Mark Goresky for lots of correspondence and conversations, especially about Thom's ``maps without blowups,'' and to Kevin McGerty for suggesting I think about induction theorems.  I have also benefited from discussions with Zhiwei Yun.

\end{document}